\documentclass{gtpart}     
\usepackage{amscd}
\usepackage{amssymb}
\usepackage{amsthm}
\usepackage{amsmath}
\numberwithin{equation}{section}
\usepackage{hyperref}
\usepackage{extarrows}
\usepackage{tikz}
\usepackage{multirow}
\usepackage{verbatim}

\usetikzlibrary{matrix}
\usetikzlibrary{arrows}
\usetikzlibrary{decorations.pathreplacing,decorations.markings}
\usetikzlibrary{matrix,arrows}
\usetikzlibrary{arrows,shapes,positioning}
\usetikzlibrary{decorations.markings}
\tikzstyle arrowstyle=[scale=1]
\tikzstyle directed=[postaction={decorate,decoration={markings,
    mark=at position .5 with {\arrow[arrowstyle]{stealth}}}}]
\tikzstyle reverse directed=[postaction={decorate,decoration={markings,
    mark=at position .65 with {\arrowreversed[arrowstyle]{stealth};}}}]

\numberwithin{equation}{section}

\title{Second Mod $2$ Homology of Artin Groups}

\author{Toshiyuki Akita}
\givenname{Toshiyuki}
\surname{Akita}
\address{Department of Mathematics, Hokkaido University, North 10,West 8, Kita-ku, Sapporo 060-0810, JAPAN}
\email{akita@math.sci.hokudai.ac.jp}
\urladdr{}

\author{Ye Liu}
\givenname{}
\surname{}
\address{}
\email{liu@math.sci.hokudai.ac.jp}
\urladdr{}

\arxivreference{}
\arxivpassword{}

\newtheorem{thm}{Theorem}[section]    
\newtheorem{cor}[thm]{Corollary}
\newtheorem{prop}[thm]{Proposition}
\newtheorem{conj}[thm]{Conjecture}
\newtheorem{lem}[thm]{Lemma}          
\theoremstyle{definition}
\newtheorem{defn}[thm]{Definition}    
\newtheorem{ex}[thm]{Example}
\newtheorem*{rem}{Remark}             
\newcommand{\G}{\Gamma}
\newcommand{\rank}{\mathrm{rank~}}

\begin{document}

\begin{abstract}    
In this paper, we compute the second mod $2$ homology of an arbitrary Artin group, without assuming the $K(\pi,1)$ conjecture. The key ingredients are (A) Hopf's formula for the second integral homology of a group and (B) Howlett's result on the second integral homology of Coxeter groups.
\end{abstract}

\maketitle


\section{Introduction}
An Artin group (or an Artin-Tits group) is a finitely presented group with at most one simple relation between a pair of generators. Examples includes finitely generated free abelian groups, free groups of finite rank, Artin's braid groups with finitely many strands and right-angled Artin groups, etc. Artin groups appear in diverse branches of mathematics such as singularity theory, low dimensional topology, geometric group theory and the theory of hyperplane arrangements, etc.

Artin groups are closely related to Coxeter groups. For a Coxeter graph $\G$ and the corresponding Coxeter system $(W(\G),S)$, we associate an Artin group $A(\G)$ obtained by, informally speaking, dropping the relations that each generator has order $2$ from the standard presentation of $W(\G)$. The symmetric group $\mathfrak{S}_n$ is the Coxeter group associated to the Coxeter graph of type $A_{n-1}$ and the braid group $Br(n)$ is the corresponding Artin group. The Coxeter group $W(\G)$ can be realized as a reflection group acting on a convex cone $U$ (called Tits cone) in $\mathbb{R}^n$ with $n=\#S$ the rank of $W$. Let $\mathcal{A}$ be the collection of reflection hyperplanes. The complement 
\[
M(\G)=(\mathrm{int}(U)+\sqrt{-1}\mathbb{R})\backslash\bigcup_{H\in\mathcal{A}}H\otimes\mathbb{C}
\]
admits the free $W(\G)$-action, and the resulting orbit space $N(\G)=M(\G)/W(\G)$ has the fundamental group isomorphic to $A(\G)$ (\cite{Lek1983}). The celebrated $K(\pi,1)$ conjecture states that $N(\G)$ is a $K(A(\G),1)$ space. See Subsection \ref{Kpi1} for a list of $\G$ for which the $K(\pi,1)$ conjecture is proved.

Existing results about (co)homology of Artin groups all focus on particular types Artin groups, for which the $K(\pi,1)$ conjecture has been proved. There are very few properties that can be said for (co)homology of {\it all} Artin groups (except for their first integral homology, which are simply abelianizations). In this paper, we compute the second mod $2$ homology of {\it all} Artin groups, without assuming an affirmative solution of the $K(\pi,1)$ conjecture. Our main tools are Hopf's formula on the second homology (or the Schur multiplier) of groups, together with Howlett's theorem (Theorem \ref{How}) on the second integral homology of Coxeter groups. We are inspired by \cite{Korkmaz2003}, where the authors computed the second integral homology of the mapping class groups of oriented surfaces using Hopf's formula.

Our main result is the following.
\begin{thm}
Let $A(\G)$ be the Artin group associated to a Coxeter graph $\G$. Then
\[
H_2(A(\G);\Z_2)\cong\Z_2^{p(\G)+q(\G)},
\]
where $p(\G)$ and $q(\G)$ are non-negative integers associated to $\G$; see Theorem \ref{CE} for definitions.
\end{thm}
As a corollary, we obtain a sufficient condition that the classifying map $c:N(\G)\to K(A(\G),1)$ induces an isomorphism
\[
c_*:H_2(N(\G);\Z)\to H_2(A(\G);\Z).
\]
Furthermore, we conclude that the induced homomorphism
\[
c_*\otimes\mathrm{id}_{\Z_2}:H_2(N(\G);\Z)\otimes\Z_2\to H_2(A(\G);\Z)\otimes\Z_2
\]
is always an isomorphism. This provides affirmative evidence for the $K(\pi,1)$ conjecture.

A part of contents of this paper is based on the second author's Ph.D. thesis.

\section{Preliminaries}

We collect relevant definitions and properties of Coxeter groups and Artin groups. We refer to \cite{Bourbaki1968, Humphreys1990} for Coxeter groups and \cite{Paris2009,Paris2012a,Paris2014} for Artin groups.

\subsection{Coxeter groups}\label{Coxeter}
Let $S$ be a finite set. A {\it Coxeter matrix} over $S$ is a symmetric matrix $M=(m(s,t))_{s,t\in S}$ such that $m(s,s)=1$ for all $s\in S$ and $m(s,t)=m(t,s)\in\{2,3,\cdots\}\cup\{\infty\}$ for distinct $s,t\in S$. It is convenient to represent $M$ by a labeled graph $\G$, called the {\it Coxeter graph} of $M$ defined as follows:
\begin{itemize}
\item The vertex set $V(\G)=S$;
\item The edge set $E(\G)=\left\{\{s,t\}\subset S\mid m(s,t)\geq 3\right\}$;
\item The edge $\{s,t\}$ is labeled by $m(s,t)$ if $m(s,t)\geq 4$.
\end{itemize}
Let $\G_{odd}$ be the subgraph of $\G$ with $V(\G_{odd})=V(\G)$ and $E(\G_{odd})=\{\{s,t\}\in E(\G)\mid m(s,t) \text{ odd}\}$ inheriting labels from $\G$. By abuse of notations, we frequently regard $\G$  (hence also $\G_{odd}$) as its underlying $1$-dimensional CW-complex.

\begin{defn}
Let $\G$ be a Coxeter graph and $S$ its vertex set. The {\it Coxeter system} associated to $\G$ is the pair $(W(\G),S)$, where the {\it Coxeter group} $W(\G)$ is defined by the following standard presentation
\[ 
W(\G)=\langle S\mid (st)^{m(s,t)}=1, \forall s,t\in S \text{ with } m(s,t)\neq\infty \rangle.
\]
Each generator $s\in S$ of $W$ has order $2$. For distinct $s,t\in S$, the order of $st$ is precisely $m(s,t)$ if $m(s,t)\neq\infty$. In case $m(s,t)=\infty$, the element $st$ has infinite order. 
\end{defn}

However, in this paper, we adopt an equivalent definition. To do so, we first introduce a notation. For two letters $s,t$ and an integer $m\geq 2$, we shall use the following notation of the word of length $m$ consisting of $s$ and $t$ in an alternating order.
\[
(st)_m:=\overbrace{sts\cdots}^m.
\]
For example, $(st)_2=st,(st)_3=sts,(st)_4=stst$.
\begin{defn}\label{cox}
Let $\G$ be a Coxeter graph and $S$ its vertex set. The {\it Coxeter group} associated to $\G$ is the group defined by the following presentation
\[
W(\G)=\langle S\mid \overline{R_W}\cup Q_W\rangle.
\]
The sets of relations are $\overline{R_W}=\{R(s,t)\mid m(s,t)<\infty\}$ and $Q_W=\{Q(s)\mid s\in S\}$, where $R(s,t):=(st)_{m(s,t)}(ts)_{m(s,t)}^{-1}$ and $Q(s):=s^2$.

Note that since $R(s,t)=R(t,s)^{-1}$, we may reduce the relation set $\overline{R_W}$ by introducing a total order on $S$ and put $R_W:=\{R(s,t)\mid m(s,t)<\infty, s<t\}$. We have the following presentation with fewer relations
\[
W(\G)=\langle S\mid R_W\cup Q_W\rangle.
\]
We shall omit the reference to $\G$ if there is no ambiguity. The rank of $W$ is defined to be $\#S$.
\end{defn}

Let $(W,S)$ be a Coxeter system. For a subset $T\subset S$, let $W_T$ denote the subgroup of $W$ generated by $T$, called a {\it parabolic subgroup} of $W$. In particular, $W_S=W$ and $W_{\varnothing}=\{1\}$. It is known that $(W_T,T)$ is the Coxeter system associated to the Coxeter graph $\G_T$ (the full subgraph of $\G$ spanned by $T$ inheriting labels)(cf. Th\'{e}or\`eme 2 in Chapter IV of \cite{Bourbaki1968}).

\subsection{Artin groups}\label{Artin}

The Artin group $A(\G)$ associated to a Coxeter graph $\G$ is obtained from the presentation of $W(\G)$ by dropping the relation set $Q_W$.

\begin{defn}\label{art}
Given a Coxeter graph $\G$ (hence a Coxeter system $(W,S)$), we introduce a set $\Sigma=\{a_s\mid s\in S\}$ in one-to-one correspondence with $S$. Then the Artin system associated to $\G$ is the pair $(A(\G),\Sigma)$, where $A(\G)$ is the {\it Artin group} of type $\G$ defined by the following presentation:
\[
A(\G)=\langle \Sigma\mid \overline{R_A}\rangle,
\]
where $\overline{R_A}=\{R(a_s,a_t)\mid m(s,t)<\infty\}$ and $R(a_s,a_t)=(a_sa_t)_{m(s,t)}(a_ta_s)_{m(s,t)}^{-1}$.

As in the Coxeter group case, we introduce a total order on $S$ and put $R_A:=\{R(a_s,a_t)\mid m(s,t)<\infty, s<t\}$. We have the following presentation with fewer relations
\[
A(\G)=\langle \Sigma\mid R_A\rangle.
\]
\end{defn}

There is a canonical projection $p: A(\G)\to W(\G)$, $a_s\mapsto s~ (s\in S)$, whose kernel is called the {\it pure Artin group} of type $\G$.

We say that an Artin group $A(\G)$ is of finite type (or spherical type) if the associated Coxeter group $W(\G)$ is finite, otherwise $A(\G)$ is of infinite type (or non-spherical type).

\subsection{$K(\pi,1)$ conjecture}\label{Kpi1}
Consider a Coxeter graph $\G$ and the associated Coxeter system $(W,S)$ with rank $\#S=n$. Recall that $W$ can be realized as a reflection group acting on a Tits cone $U\subset\mathbb{R}^n$ (see \cite{Paris2012a}). Let $\mathcal{A}$ be the collection of the reflection hyperplanes. Put
\[
M(\G):=\left(\mathrm{int}(U)+\sqrt{-1}\mathbb{R}^n\right)\backslash\bigcup_{H\in\mathcal{A}}H\otimes\mathbb{C}.
\]
Then $W$ acts on $M(\G)$ freely and properly discontinuously. Denote the orbit space by
\begin{equation}\label{N}
N(\G):=M(\G)/W.
\end{equation}
It is known that 
\begin{thm}[\cite{Lek1983}]
The fundamental group of $N(\G)$ is isomorphic to the Artin group $A(\G)$.
\end{thm}
In general, $N(\G)$ is only conjectured to be a classifying space of $A(\G)$.
\begin{conj}
Let $\G$ be an arbitrary Coxeter graph, then the orbit space $N(\G)$ is a $K(\pi,1)$ space, hence is a classifying space of the Artin group $A(\G)$.
\end{conj}

This conjecture is proved to hold for a few classes of Artin groups. Here is a list of such classes known so far.
\begin{itemize}
\item Artin groups of finite type (\cite{Deligne1972}).
\item Artin groups of large type (\cite{Hendriks1985}).
\item $2$-dimensional Artin groups (\cite{Charney1995}).
\item Artin groups of FC type (\cite{Charney1995}).
\item Artin groups of affine types $\widetilde{A_n},\widetilde{C_n}$ (\cite{Okonek1979}).
\item Artin groups of affine type $\widetilde{B_n}$ (\cite{Callegaro2010}).
\item Artin group $A(\G)$ such that the $K(\pi,1)$ conjecture holds for all $A(\G_T)$ where $T\subset S$ and $\G_T$ does not contain $\infty$-labeled edges (\cite{Ellis2010}).
\end{itemize}

\subsection{First and second homology of $N(\G)$}\label{1&2}
Clancy and Ellis \cite{Clancy2010} computed the second integral homology of $N(\G)$ using the Salvetti complex for an Artin group. We recall their result and follow their notations.

Let us first fix some notations. Let $\G$ be a Coxeter graph with vertex set $S$. Define $Q(\G)=\{\{s,t\}\subset S\mid m(s,t) \text{ is even}\}$ and $P(\G)=\{\{s,t\}\subset S\mid m(s,t)=2\}$. Write $\{s,t\}\equiv\{s^{\prime},t^{\prime}\}$ if two such pairs in $P(\G)$ satisfy $s=s^{\prime}$ and $m(t,t^{\prime})$ is odd. This generates an equivalence relation on $P(\G)$, denoted by $\sim$. Let $P(\G)/\sim$ be the set of equivalence classes. An equivalence class is called a {\it torsion} class if it is represented by a pair $\{s,t\}\in P(\G)$ such that there exists a vertex $v\in S$ with $m(s,v)=m(t,v)=3$. In the above situation, Clancy and Ellis proved the following theorem.

\begin{thm}[\cite{Clancy2010}]\label{CE}
Let $\G$ be a Coxeter graph and $N(\G)$ as in (\ref{N}), then
\[
H_2(N(\G);\Z)\cong\Z_2^{p(\G)}\oplus\Z^{q(\G)},
\]
where
\begin{align*}
p(\G)&:=\text{ number of torsion classes in }P(\G)/\sim,\\
q_1(\G)&:=\text{ number of non-torsion classes in }P(\G)/\sim,\\
q_2(\G)&:=\#(Q(\G)-P(\G))=\#\{\{s,t\}\subset S\mid m(s,t)\geq 4 \text{ is even}\},\\
q_3(\G)&:=\rank H_1(\G_{odd};\Z),\\
q(\G)&:=q_1(\G)+q_2(\G)+q_3(\G).
\end{align*}
\end{thm}

\begin{rem}\label{depend}
Note that $H_1(N(\G);\Z)\cong H_1(A(\G);\Z)$ is isomorphic to the abelianization of $A(\G)$, which is a free abelian group with rank equals to $\rank H_0(\G_{odd};\Z)$, the number of connected components of $\G_{odd}$.
\end{rem}

\section{Second mod 2 homology of Artin groups}\label{mainresult}
The (co)homology of the orbit space $N(\G)$ coincides with that of the Artin group $A(\G)$, provided the $K(\pi,1)$ conjecture for $A(\G)$ holds. There are many results about (co)homology of $N(\G)$ in the literature, for example \cite{DeConcini1999a,DeConcini2001,Callegaro2008a,Callegaro2010}. The $K(\pi,1)$ conjecture is known to hold in these cases.

In this section, nevertheless, we shall work on the second homology of {\it arbitrary} Artin groups, without assuming that the $K(\pi,1)$ conjecture holds. Our main result is the following theorem.

\begin{thm}\label{H2Z2}
Let $\G$ be an arbitrary Coxeter graph and $A(\G)$ the associated Artin group. Then the second mod $2$ homology of $A(\G)$ is
\[
H_2(A(\G);\Z_2)\cong\Z_2^{p(\G)+q(\G)},
\]
where $p(\G)$ and $q(\G)$ are as in Theorem \ref{CE}.
\end{thm}

The outline of our proof is as follows. In Subsection \ref{HCE}, we state Howlett's theorem on the second integral homology group $H_2(W(\G);\Z)$ of the Coxeter group $W(\G)$. Next in Subsection \ref{HF}, we recall Hopf's formula of the second homology of a group. The key of the proof is that, by virtue of Hopf's formula, we are able to find explicitly a set $\Omega(W)$ of generators of $H_2(W(\G);\Z)$ (Subsection \ref{HtoC}), as well as a set $\Omega(A)$ of generators of $H_2(A(\G);\Z)$ (Subsection \ref{HtoA}). On the other hand, Howlett's theorem implies that $\Omega(W)$ forms a basis of $H_2(W(\G);\Z)$, which is an elementary abelian $2$-group of rank $p(\G)+q(\G)$. Furthermore, we will show that the homomorphism $p_*:H_2(A(\G);\Z)\to H_2(W(\G);\Z)$ induced by the projection $p:A(\G)\to W(\G)$ maps $\Omega(A)$ onto $\Omega(W)$. Hence $p_*$ is actually an epimorphism and becomes an isomorphism when tensored with $\Z_2$.

\subsection{Howlett's theorem}\label{HCE}
As mentioned in the previous paragraph, we shall study the homomorphim $p_*:H_2(A(\G);\Z)\to H_2(W(\G);\Z)$ induced by the projection $p:A(\G)\to W(\G)$. A reason for doing so is that we have the following Howlett's theorem.

\begin{thm}[\cite{Howlett1988}]\label{How}
The second integral homology of the Coxeter group $W(\G)$ associated to a Coxeter graph $\G$ is
\[
H_2(W(\G);\Z)\cong\Z_2^{p(\G)+q(\G)},
\]
where $p(\G)$ and $q(\G)$ are as in Theorem \ref{CE}.
\end{thm}

\begin{rem}
The original statement in \cite{Howlett1988} was 
\[
H_2(W(\G);\Z)\cong\Z_2^{-n_1(\G)+n_2(\G)+n_3(\G)+n_4(\G)},
\]
where
\begin{align*}
n_1(\G)&:=\#S,\\
n_2(\G)&:=\#\{\{s,t\}\in E(\G)\mid m(s,t)<\infty\},\\
n_3(\G)&:=\#P(\G)/\sim,\\
n_4(\G)&:=\rank H_0(\G_{odd};\Z).
\end{align*}
For a Coxeter graph $\G$, the above numbers are related to those used by Clancy-Ellis as follows
\[
-n_1(\G)+n_2(\G)+n_3(\G)+n_4(\G)=p(\G)+q(\G).
\]
In fact, $n_1(\G)=\#V(\G_{odd})$, $n_2(\G)=q_2(\G)+\# E(\G_{odd})$ and $n_3(\G)=p(\G)+q_1(\G)$. The above equation follows from the Euler-Poincar\'e theorem applied to $\G_{odd}$,
\[
\#V(\G_{odd})-\#E(\G_{odd})=\rank H_0(\G_{odd};\Z)-\rank H_1(\G_{odd};\Z).
\]
\end{rem}


\begin{ex}\label{di}
We shall make use of the following example later.

Let $\G=I_2(m)$. Thus $W(\G)=D_{2m}$ is the dihedral group of order $2m$. Theorem \ref{How} shows that
\begin{eqnarray*}
H_2(W(\G);\Z)\cong
\begin{cases}
\Z_2, &m \text{ is even};\cr
0, &m \text{ odd}.
\end{cases}
\end{eqnarray*}
See Corollary 10.1.27 of \cite{Karpilovsky1993} for a complete list of integral homology of dihedral groups.
\end{ex}

\subsection{Hopf's formula}\label{HF}
Hopf's formula gives a description of the second integral homology of a group. We first recall some notations. For a group $G$, the commutator of $x,y\in G$ is the element $[x,y]=xyx^{-1}y^{-1}$. The commutator subgroup $[G,G]$ of $G$ is the subgroup of $G$ generated by all commutators. In general, we define $[H,K]$ as the subgroup of $G$ generated by $[h,k],h\in H, k\in K$ for any subgroups $H$ and $K$ of $G$.
\begin{thm}[Hopf's formula]\label{Hopf}
If a group $G$ has a presentation $\langle S\mid R\rangle$, then
\[
H_2(G;\Z)\cong\frac{N\cap[F,F]}{[F,N]},
\]
where $F=F(S)$ is the free group generated by $S$ and $N=N(R)$ is the normal closure of $R$ (subgroup of $F$ normally generated by the relation set $R$).
\end{thm}
See Section II.5 of \cite{Brown1982} for a topological proof. Moreover, Hopf's formula admits the following naturality (see Section II.6, Exercise 3(b) of \cite{Brown1982}).
\begin{prop}\label{nat}
	Let $G=F/N=\langle S\mid R\rangle$ and $G^{\prime}=F^{\prime}/N^{\prime}=\langle S^{\prime}\mid R^{\prime}\rangle$ as in Theorem \ref{Hopf}. Suppose a homomorphism $\alpha:G\to G^{\prime}$ lifts to $\widetilde{\alpha}:F\to F^{\prime}$. Then the following diagram commutes,
	\[
\begin{tikzpicture}
  \matrix (m) [matrix of math nodes,row sep=3em,column sep=4em,minimum width=2em]
  {
     H_2(G;\Z) & N\cap[F,F/[F,N] \\
     H_2(G^{\prime};\Z) & N^{\prime}\cap[F^{\prime},F^{\prime}]/[F^{\prime},N^{\prime}] \\};
  \path[-stealth]
    (m-1-1) edge node [left] {$H_2(\alpha)$} (m-2-1)
    (m-1-1) edge node [below] {$\cong$} (m-1-2)
    (m-2-1) edge node [below] {$\cong$} (m-2-2)
    (m-1-2) edge node [right] {$\alpha_*$} (m-2-2);
\end{tikzpicture}
\]
where $\alpha_*$ is induced by $\widetilde{\alpha}$.
\end{prop}

For simplicity we denote by $\langle x\rangle_G=x[F,N]\in F/[F,N]$ the coset of $[F,N]$ represented by $x\in F$ and $\langle x,y\rangle_G=[x,y][F,N]\in [F,F]/[N,F]$ for $x,y\in F$. Thanks to Hopf's formula, second homology classes of $G$ can be regarded as $\langle x\rangle_G$ for $x\in N\cap [F,F]$.

To see how the representatives look like, we make the following simple observations, which we learned from \cite{Korkmaz2003}.
\begin{lem}\label{abelian}
The group $N/[F,N]$ is abelian.
\end{lem}
\begin{proof}
Note that $N/[F,N]$ is a quotient group of $N/[N,N]$ and the latter is the abelianization of $N$.
\end{proof}
Thus we write the group $N/[F,N]$ additively. It is clear $\langle n\rangle_G=-\langle n^{-1}\rangle_G$ for $n\in N$.
\begin{lem}\label{conj}
In the abelian group $N/[F,N]$, we have
\[
\langle n\rangle_G=\langle fnf^{-1}\rangle_G
\]
for $n\in N$ and $f\in F$.
\end{lem}
\begin{proof}
Since $[f,n]\in[F,N]$, $\langle f,n\rangle_G=\langle fnf^{-1}n^{-1}\rangle_G=\langle fnf^{-1}\rangle_G-\langle n\rangle_G=0$.
\end{proof}

Therefore a coset in $N/[F,N]$ is represented by an element of the form $\prod_{r\in R}r^{n(r)}~(n(r)\in\Z)$. Hopf's formula implies that a second homology class of $G$ can be represented by an element $\prod_{r\in R}r^{n(r)}\in[F,F]$.

The next lemma is useful.
\begin{lem}\label{comm}
Let $G=F/N$ be as in Theorem \ref{Hopf}. If $x,y,z\in F$ such that $[x,y],[x,z]\in N\cap[F,F]$, then
\[
\langle x, yz\rangle_G=\langle x,y\rangle_G+\langle x,z\rangle_G,~~~ \langle x,y^{-1}\rangle_G=-\langle x,y\rangle_G,
\]
\end{lem}
\begin{proof}
Note that $[x,yz]=[x,y]y[x,z]y^{-1}$. Then in the abelian group $N/[F,N]$,
\[
\langle x,yz\rangle_G=\langle x,y\rangle_G+\langle y[x,z]y^{-1}\rangle_G.
\]
The term $\langle y[x,z]y^{-1}\rangle_G=\langle x,z\rangle_G$ since
\[
[x,z]^{-1}y[x,z]y^{-1}=[[x,z]^{-1},y]\in [N,F].
\]
Hence the first equality holds. The second follows immediately from the first.
\end{proof}

\subsection{Hopf's formula applied to Coxeter groups}\label{HtoC}
The aim of this subsection is to construct an explicit set $\Omega(W)$ of generators of $H_2(W(\G);\Z)$. Combined with Howlett's theorem (Theorem \ref{How}), we show that $\Omega(W)$ is a basis of $H_2(W(\G);\Z)$.

Let us describe the construction of $\Omega(W)$. Let $\G$ be a Coxeter graph and $(W,S)$ the associated Coxeter system with $S$ totally ordered. Then $W=\langle S\mid R_W\cup Q_W\rangle$ is as in Definition \ref{cox}. Let $F_W=F(S)$ be the free group on $S$ and $N_W=N(R_W\cup Q_W)$ be the normal closure of $R_W\cup Q_W$. Therefore $W=F_W/N_W$. Using Hopf's formula we identify $H_2(W;\Z)\cong (N_W\cap[F_W,F_W])/[F_W,N_W]$. We shall construct three sets $\Omega_i(W)\subset (N_W\cap[F_W,F_W])/[F_W,N_W]~(i=1,2,3)$. In view of Lemma \ref{abelian} and Lemma \ref{conj}, a second homology class of $W$ is of the form $\langle x\rangle_W$ with $x$ expressed by a word $\prod_{R(s,t)\in R_W}R(s,t)^{n(s,t)}\prod_{Q(s)\in Q_W}Q(s)^{n(s)}\in[F_W,F_W]$. We decompose $\langle x\rangle_W=\langle x_1\rangle_W+\langle x_2\rangle_W+\langle x_3\rangle_W$ as in the proof of Theorem \ref{OmegaW}, such that $\langle x_i\rangle_W$ is generated by $\Omega_i(W)$. Then $\Omega(W)=\Omega_1(W)\cup\Omega_2(W)\cup\Omega_3(W)$ generates $H_2(W;\Z)$. Now we exhibit respectively the constructions of $\Omega_i(W)~(i=1,2,3)$.

\subsubsection{Construction of $\Omega_1(W)$}
Let
\[
\Omega_1(W)=\{\langle s,t\rangle_W\mid s,t\in S, s<t, m(s,t)=2\}.
\]
Recall that $\langle s,t\rangle_W=[s,t][F_W,N_W]\in (N_W\cap[F_W,F_W])/[F_W,N_W]$ and $R(s,t)=[s,t]$ when $m(s,t)=2$. Note that the above expression may have repetitions. In fact, we have the following.
\begin{prop}
    $\#\Omega_1(W)\leq p(\G)+q_1(\G)$.
\end{prop}
\begin{proof}
	We shall show that $\langle s,t \rangle_W=\langle s,t^{\prime}\rangle_W$ in $\Omega_1(W)$ if $\{s,t\}\equiv\{s,t^{\prime}\}$ in $P(\G)$. 
	Suppose $s<t$ and $s<t^{\prime}$ with $\{s,t\}\equiv\{s,t^{\prime}\}$ in $P(\G)$, that is $m(s,t)=m(s,t^{\prime})=2$ and $m(t,t^{\prime})$ is odd. Then in $N_W/[F_W,N_W]$,
\begin{align*}
\langle s,t\rangle_W-\langle s,t^{\prime}\rangle_W&=\langle s,R(t,t^{\prime})\rangle_W\\
&=\langle sR(t,t^{\prime})s^{-1}R(t,t^{\prime})^{-1}\rangle_W\\
&=\langle sR(t,t^{\prime})s^{-1}\rangle_W+\langle R(t,t^{\prime})^{-1}\rangle_W\\
&=\langle R(t,t^{\prime})\rangle_W-\langle R(t,t^{\prime})\rangle_W=0,
\end{align*}
where the first and the third equalities follow from Lemma \ref{comm}, the fourth from Lemma \ref{conj}. Similarly, $\langle s,t \rangle_W=\langle s^{\prime},t\rangle_W$ in $\Omega_1(W)$ if $\{s,t\}\equiv\{s^{\prime},t\}$ in $P(\G)$. Hence $\#\Omega_1(W)\leq \#\left(P(\G)/\sim\right)=p(\G)+q_1(\G)$.
\end{proof}

\subsubsection{Construction of $\Omega_2(W)$}
Let
\[
\Omega_2(W)=\{\langle R(s,t)\rangle_W\mid s,t\in S, s<t, m(s,t)\geq 4 \text{ is even}\}.
\]
Recall that $R(s,t)=(st)_{m(s,t)}(ts)_{m(s,t)}^{-1}$. Note that when $m(s,t)$ is even,  $R(s,t)$ is in the kernel of the abelianization map $\mathrm{Ab}:F_W\to F_W/[F_W,F_W]$ and hence $R(s,t)\in [F_W,F_W]$. The following is an obvious observation.
\begin{prop}
	$\#\Omega_2(W)\leq q_2(\G)$.
\end{prop}

\subsubsection{Construction of $\Omega_3(W)$}
The construction of $\Omega_3(W)$ requires more preparations. Recall that $\G_{odd}$ is the subgraph of $\G$ considered as a $1$-dimensional CW-complex with $0$-cells $S$ and $1$-cells $\{\langle s,t\rangle\mid s,t\in S, s<t, m(s,t) \text{ odd}\}$ oriented by $\partial\langle s,t\rangle=t-s$. We define a group 
\[
\mathcal{C}_W=\{(\alpha,\beta)\in C_1(\G_{odd})\oplus 2C_0(\G_{odd})\mid \partial\alpha=\beta\}
\]
where $2C_0(\G_{odd})=\{2\gamma\mid \gamma\in C_0(\G_{odd})\}$ is the group of $0$-chains with all coefficients even, 
and a subgroup $\mathcal{D}_W$ of $\mathcal{C}_W$ generated by $(2\langle s,t\rangle,-2s+2t)$ for all $1$-cells $\langle s,t\rangle$.

Consider the following homomorphism
\[
\Phi_W:\mathcal{C}_W\to \frac{N_W\cap[F_W,F_W]}{[F_W,N_W]}
\]
defined by
\[
\Phi_W\left(\sum_{s<t,~m(s,t) \text{ odd}}n(s,t)\langle s,t\rangle,\sum_{s\in S}2n(s) s\right)=\left\langle\prod_{s<t,~m(s,t) \text{ odd}}R(s,t)^{n(s,t)}\prod_{s\in S}Q(s)^{n(s)}\right\rangle_W.
\]
The definition is indeed valid by the following easy lemma.
\begin{lem}\label{x3}
The following are equivalent.

(A) $\left(\sum_{s<t,~m(s,t) \text{ odd}}n(s,t)\langle s,t\rangle,\sum_{s\in S}2n(s) s\right)\in \mathcal{C}_W$.\\
(B) $\prod_{s<t,~m(s,t) \text{ odd}}R(s,t)^{n(s,t)}\prod_{s\in S}Q(s)^{n(s)}\in[F_W,F_W]$. 
\end{lem}
\begin{proof}
We suppress the ranges since they should be clear.
	\begin{align*}
		\text{(A)}&\Leftrightarrow \partial\left(\sum n(s,t)\langle s,t\rangle\right)=\sum 2n(s) s\\
		&\Leftrightarrow \sum 2n(s) s+\sum n(s,t)(s-t)=0\\
		&\Leftrightarrow \mathrm{Ab}\left(\prod R(s,t)^{n(s,t)}\prod Q(s)^{n(s)}\right)=0\Leftrightarrow \text{(B)},		
	\end{align*}
	where $\mathrm{Ab}:F_W\to F_W/[F_W,F_W]$ is the abelianization map and we write $F_W/[F_W,F_W]$ additively. Note that $\mathrm{Ab}(R(s,t))=s-t$ if $m(s,t)$ is odd.
\end{proof}

The following is a consequence of Example \ref{di}.
\begin{prop}
$\mathcal{D}_W$ lies in the kernel of $\Phi_W$.
\end{prop}
\begin{proof}
It suffices to show that any generator $(2\langle s,t\rangle,-2s+2t)$ of $\mathcal{D}_W$ is mapped to the identity by $\Phi_W$, or equivalently, the word $\left((st)_m(ts)_m^{-1}\right)^2(s^2)^{-1}t^2$ lies in $[F_W,N_W]$ when $m$ is odd. Let $s,t\in S$ with $m:=m(s,t)$ odd, consider the parabolic subgroup $W^{\prime}:=W_{\{s,t\}}$ of $W$, which is isomorphic to the dihedral group $D_{2m}$ of order $2m$. From Example \ref{di}, we know $H_2(W^{\prime};\Z)=0$. On the other hand, Hopf's formula applied to $W^{\prime}$ shows that $H_2(W^{\prime};\Z)\cong (N_{W^{\prime}}\cap [F_{W^{\prime}},F_{W^{\prime}}])/[F_{W^{\prime}},N_{W^{\prime}}]$. Therefore the word $\left((st)_m(ts)_m^{-1}\right)^2(s^2)^{-1}t^2\in N_{W^{\prime}}\cap [F_{W^{\prime}},F_{W^{\prime}}]$ represents the trivial homology class. That is to say
\[
\left((st)_m(ts)_m^{-1}\right)^2(s^2)^{-1}t^2\in[F_{W^{\prime}},N_{W^{\prime}}]\subset[F_W,N_W].
\]
This proves the proposition.
\end{proof}

As a consequence, the homomorphism $\Phi_W$ factors through
\[
\mathcal{C}_W\twoheadrightarrow\mathcal{C}_W/\mathcal{D}_W\to (N_W\cap[F_W,F_W])/[F_W,N_W]. 
\]
Let $Z_1(\G_{odd};\Z)$ denote the group of $1$-cycles of $\G_{odd}$ with integral coefficients and $Z_1(\G_{odd};\Z_2)$ the group of $1$-cycles of $\G_{odd}$ with coefficients in $\Z_2$. 
Define a homomorphism $\Xi_W:\mathcal{C}_W\to Z_1(\G_{odd};\Z_2)$ by $(\alpha,\partial\alpha)\mapsto\overline{\alpha}$, where $\alpha\in C_1(\G_{odd})$ such that $\partial{\alpha}\in 2C_0(\G_{odd})$ and $\overline{\alpha}\in C_1(\G_{odd};\Z_2)$ is the mod $2$ reduction of $\alpha$. The condition $\partial\alpha\in 2C_0(\G_{odd})$ asserts that $\overline{\alpha}$ is indeed a $1$-cycle of $\G_{odd}$ with coefficients in $\Z_2$.

\begin{prop}\label{H1Godd}
The homomorphism $\Xi_W:\mathcal{C}_W\to Z_1(\G_{odd};\Z_2)$ factors through an isomorphism
\[
\begin{tikzpicture}
\matrix(m)[matrix of math nodes,
row sep=2.6em, column sep=2em,
text height=1.5ex, text depth=0.25ex]
{\mathcal{C}_W& \\
\mathcal{C}_W/\mathcal{D}_W&Z_1(\G_{odd};\Z_2)\\};
\path[->,font=\scriptsize,>=angle 90]
(m-1-1) edge node[auto] {} (m-2-1)
        edge node[auto] {$\Xi_W$} (m-2-2)
(m-2-1) edge node[auto] {$\cong$} (m-2-2);
\end{tikzpicture}
\]
\end{prop}
\begin{proof}
The homomorphism $\Xi_W$ is obviously an epimorphism and
	\begin{align*}
		(\alpha,\partial\alpha)\in\mathrm{Ker}~\Xi_W &\Leftrightarrow \overline{\alpha}=0\in Z_1(\G_{odd};\Z_2)\\
		&\Leftrightarrow \alpha\in 2C_1(\G_{odd})\\
		&\Leftrightarrow (\alpha,\partial\alpha)\in\mathcal{D}_W.		
	\end{align*}
	Hence $\mathrm{Ker}~\Xi_W=\mathcal{D}_W$.
\end{proof}

Via the isomorphism in Proposition \ref{H1Godd}, we obtain a homomorphism
\[
\Psi_W: Z_1(\G_{odd};\Z_2)\to \frac{N_W\cap[F_W,F_W]}{[F_W,N_W]},
\]
which fits into the following commutative diagram
\begin{equation}\label{PsiW}
\begin{tikzpicture}
\matrix(m)[matrix of math nodes,
row sep=2.6em, column sep=2em,
text height=1.5ex, text depth=0.25ex]
{\mathcal{C}_W& N_W\cap[F_W,F_W]/[F_W,N_W]\\
\mathcal{C}_W/\mathcal{D}_W&Z_1(\G_{odd};\Z_2)\\};
\path[->,font=\scriptsize,>=angle 90]
(m-1-1) edge node[auto] {$\Phi_W$} (m-1-2)
        edge node[auto] {} (m-2-1)
(m-2-2) edge node[auto] {$\Psi_W$} (m-1-2)
(m-2-1) edge node[auto] {$\cong$} (m-2-2)
        edge node[auto] {} (m-1-2);
\end{tikzpicture}
\end{equation}

We fix a basis $\Omega(\G_{odd};\Z)$ of $Z_1(\G_{odd};\Z)\cong\Z^{q_3(\G)}$ once and for all and denote by $\Omega(\G_{odd};\Z_2)$ the basis of $Z_1(\G_{odd};\Z_2)\cong\Z_2^{q_3(\G)}$ obtained from $\Omega(\G_{odd};\Z)$ by mod $2$ reduction $Z_1(\G_{odd};\Z)\twoheadrightarrow Z_1(\G_{odd};\Z_2)$.

Define $\Omega_3(W)$ to be the image of $\Omega(\G_{odd};\Z_2)$ under $\Psi_W$,
\[
\Omega_3(W)=\Psi_W(\Omega(\G_{odd};\Z_2)).
\]
To be precise,
\[
\Omega_3(W)=\left\{\left\langle \prod_{s<t,m(s,t) \text{ odd}}R(s,t)^{n(s,t)}\right\rangle_W \middle| \sum_{s<t, m(s,t) \text{ odd}}n(s,t)\langle s,t\rangle\in\Omega(\G_{odd};\Z_2)\right\}
\]
\begin{prop}
	$\#\Omega_3(W)\leq q_3(\G)$.
\end{prop}

Let $\Omega(W)=\Omega_1(W)\cup\Omega_2(W)\cup\Omega_3(W)$, we conclude that
\begin{thm}\label{OmegaW}
	$\Omega(W)$ is a basis of $H_2(W;\Z)$.
\end{thm}
\begin{proof}
	Since $\#\Omega(W)\leq p(\G)+q(\G)$ and $H_2(W;\Z)\cong\Z_2^{p(\G)+q(\G)}$ (Theorem \ref{How}). It suffices to show that $\Omega(W)$ generates $H_2(W;\Z)$. An arbitrary homology class in $H_2(W;\Z)$ is represented by the coset $\langle x_1x_2x_3\rangle_W$ with
	\begin{align*}
		x_1&=\prod_{s<t,~m(s,t)=2}[s,t]^{n(s,t)},\\
		x_2&=\prod_{s<t,~m(s,t)\geq 4 \text{ is even}}R(s,t)^{n(s,t)},\\
		x_3&=\prod_{s<t,~m(s,t) \text{ odd}}R(s,t)^{n(s,t)}\prod_{Q(s)\in Q_W}Q(s)^{n(s)}.\\
	\end{align*}
    with $x_1,x_2,x_3\in [F_W,F_W]$. Thus $\langle x_1x_2x_3\rangle_W=\langle x_1\rangle_W+\langle x_2\rangle_W+\langle x_3\rangle_W$. We claim that $\langle x_i\rangle_W$ is generated by $\Omega_i(W)$. In fact, the claim for $i=1,2$ is straightforward.  For $i=3$, let 
    \[
    \alpha=\sum_{s<t,~m(s,t) \text{ odd}}n(s,t)\langle s,t\rangle, ~~~\beta=\sum_{s\in S}2n(s) s.
    \]
   Thus $(\alpha,\beta)\in\mathcal{C}_W$ by Lemma \ref{x3} with $\Phi_W(\alpha,\beta)=\langle x_3\rangle_W$. By the commutative diagram (\ref{PsiW}), the mod $2$ reduction $\overline{\alpha}\in Z_1(\G_{odd};\Z_2)$ of $\alpha$ is mapped to $\langle x_3\rangle_W$ by $\Psi_W$. This proves the claim.
\end{proof}
\begin{rem}
It is worth noting that in the previous proof, we have managed to get rid of the relations $Q(s)$	without altering the homology class $\langle x_3\rangle_W$. This will be crucial in the proof of Theorem \ref{epi}.
\end{rem}

\subsection{Hopf's formula applied to Artin groups}\label{HtoA}
Now we turn to the Artin group case. The arguments here are parallel to those in the Coxeter group case.

Let $\G$ be a Coxeter graph with the vertex set $S$ totally ordered, $A=A(\G)$ be the Artin group of type $\G$ with the presentation $A=\langle\Sigma\mid R_A\rangle$ given in Definition \ref{art}. Let $F_A=F(\Sigma)$ be the free group on $\Sigma$ and $N_A$ be the normal closure of $R_A$. Hopf's formula yields $H_2(A;\Z)\cong(N_A\cap[F_A,F_A])/[F_A,N_A]$. For the same reason as before, a second homology class of $A$ is represented by a coset $\langle x\rangle_A$ with $x$ of the form $\prod_{R(a_s,a_t)\in R_A}R(a_s,a_t)^{n(s,t)}\in[F_A,F_A]$.

We construct a set $\Omega(A)$ of generators of $H_2(A;\Z)$ using the same method as in the previous subsection.

\subsubsection{Constructions of $\Omega_1(A)$ and $\Omega_2(A)$}
The constructions of $\Omega_1(A)$ and $\Omega_2(A)$ are exactly parallel to those in the Coxeter case. Let
\begin{align*}
	\Omega_1(A)&=\{\langle a_s,a_t\rangle_A \mid s,t\in S, s<t, m(s,t)=2\},\\
	\Omega_2(A)&=\{\langle R(a_s,a_t)\rangle_A \mid s,t\in S, s<t, m(s,t)\geq4 \text{ is even}\}.
\end{align*}
The same reasoning shows
\begin{prop}
	$\#\Omega_1(A)\leq p(\G)+q_1(\G), ~\#\Omega_2(A)\leq q_2(\G)$.
\end{prop}

\subsubsection{Construction of $\Omega_3(A)$}
Consider the following homomorhpism
\[
\Psi_A:Z_1(\G_{odd};\Z)\to\frac{N_A\cap[F_A,F_A]}{[F_A,N_A]},
\]
defined by
\[
\Psi_A\left(\sum_{s<t,m(s,t) \text{ odd}}n(s,t)\langle s,t\rangle\right)=\left\langle\prod_{s<t,m(s,t) \text{ odd}}R(a_s,a_t)^{n(s,t)}\right\rangle_A.
\] 
The definition is valid by the following lemma.
\begin{lem}
	The following are equivalent.
	
    (A) $\sum_{s<t,m(s,t) \text{ odd}}n(s,t)\langle s,t\rangle\in Z_1(\G_{odd};\Z)$.\\
    (B) $\prod_{s<t,m(s,t) \text{ odd}}R(a_s,a_t)^{n(s,t)}\in[F_A,F_A]$.
\end{lem}
\begin{proof}
	We suppress again the ranges.
	\begin{align*}
		\text{(A)}&\Leftrightarrow \partial\left(\sum n(s,t)\langle s,t\rangle\right)=0\\
		&\Leftrightarrow \sum n(s,t)(t-s)=0\\
		&\Leftrightarrow \mathrm{Ab}\left(\prod R(a_s,a_t)^{n(s,t)}\right)=0\Leftrightarrow \text{(B)},
	\end{align*}
	where $\mathrm{Ab}:F_A\to F_A/[F_A,F_A]$ is the abelianization map. Note that $\mathrm{Ab}(R(a_s,a_t))=a_s-a_t$ if $m(s,t)$ is odd.
\end{proof}
Recall that we have chosen a basis $\Omega(\G_{odd};\Z)$ for $Z_1(\G_{odd};\Z)$. Let $\Omega_3(A)$ be the image of $\Omega(\G_{odd};\Z)$ under $\Psi_A$,
\[
\Omega_3(A)=\Psi_A(\Omega(\G_{odd};\Z)).
\]
To be precise,
\[
\Omega_3(A)=\left\{\left\langle \prod_{s<t,m(s,t) \text{ odd}}R(a_s,a_t)^{n(s,t)}\right\rangle_A \middle| \sum_{s<t, m(s,t) \text{ odd}}n(s,t)\langle s,t\rangle\in\Omega(\G_{odd};\Z)\right\}
\]\begin{prop}
	$\#\Omega_3(A)\leq q_3(\G)$.
\end{prop}

Let $\Omega(A)=\Omega_1(A)\cup\Omega_2(A)\cup\Omega_3(A)$, hence $\#\Omega(A)\leq p(\G)+q(\G)$. We have the following
\begin{thm}
	$\Omega(A)$ is a set of generators of $H_2(A;\Z)$.
\end{thm}
\begin{proof}
	The proof is similar to that of Theorem \ref{OmegaW} so we omit it.
\end{proof}

\subsection{Proof of main results}
Theorem \ref{H2Z2} will follow from the next more precise theorem.
\begin{thm}\label{epi}
The projection $p:A(\G)\to W(\G)$ induces an epimorphism between the second integral homology
\[
p_*:H_2(A(\G);\Z)\twoheadrightarrow H_2(W(\G);\Z).
\]
\end{thm}

\begin{proof}
	The epimorphism $p:A\to W$ defined by $p(a_s)=s$ lifts to $\widetilde{p}:F_A\to F_W$. Then by Proposition \ref{nat}, we obtain the explicit formulation of $p_*$,
\begin{align*}
	p_*:H_2(A;\Z)\cong\frac{N_A\cap[F_A,F_A]}{[F_A,N_A]}&\to \frac{N_W\cap[F_W,F_W]}{[F_W,N_W]}\cong H_2(W;\Z)\\
	\left\langle\prod_{R(a_s,a_t)\in R_A}R(a_s,a_t)^{n(s,t)}\right\rangle_A&\mapsto \left\langle\prod_{R(s,t)\in R_W}R(s,t)^{n(s,t)}\right\rangle_W
\end{align*}
We claim that $p_*$ maps $\Omega_i(A)$ onto $\Omega_i(W)$. The claim is obvious for $i=1,2$. As for the case $i=3$, consider the following diagram
\[
\begin{tikzpicture}
\matrix(m)[matrix of math nodes,
row sep=2.6em, column sep=2em,
text height=1.5ex, text depth=0.25ex]
{\Omega(\G_{odd};\Z)&\Omega_3(A)\\
\Omega(\G_{odd};\Z_2)&\Omega_3(W)\\};
\path[->,font=\scriptsize,>=angle 90]
(m-1-1) edge node[auto] {$\Psi_A$} (m-1-2)
        edge node[auto] {mod $2$} (m-2-1)
(m-1-2) edge node[auto] {$p_*$} (m-2-2)
(m-2-1) edge node[auto] {$\Psi_W$} (m-2-2);
\end{tikzpicture}
\]
Take $\alpha=\sum_{s<t, m(s,t) \text{ odd}}n(s,t)\langle s,t\rangle\in \Omega(\G_{odd};\Z)$, then
\[
p_*\circ\Psi_A(\alpha)=\left\langle \prod_{s<t,m(s,t) \text{ odd}}R(s,t)^{n(s,t)}\right\rangle_W\in\frac{N_W\cap[F_W,F_W]}{[F_W,N_W]}.
\]
Recall the construction of $\Phi_W$, we have $\Phi_W(\alpha,\partial\alpha)=p_*\circ\Psi_A(\alpha)$. Thus we obtain
\[
\Psi_W(\overline{\alpha})=p_*\circ\Psi_A(\alpha),
\]
by the commutative diagram (\ref{PsiW}). This proves that $p_*$ maps $\Omega_3(A)$ into $\Omega_3(W)$ and the above diagram commutes. Since the mod $2$ reduction restricts to a bijection $\Omega(\G_{odd};\Z)\to \Omega(\G_{odd};\Z_2)$ and by definition the horizontal maps are onto, $p_*:\Omega_3(A)\to\Omega_3(W)$ is onto. The proof is complete.
\end{proof}

\begin{proof}[Proof of Theorem \ref{H2Z2}]
Consider the following composition of epimorphisms
\begin{equation}\label{H2AZ}
\Z^{\Omega(A)}\xlongrightarrow{\phi} H_2(A;\Z)\xlongrightarrow{p_*} H_2(W;\Z),	
\end{equation}	
where $\Z^{\Omega(A)}$ is the free abelian group generated by $\Omega(A)$ and $\phi(\omega)=\omega$ for $\omega\in\Omega(A)$. Taking tensor product with $\Z_2$ for terms in the above sequence (\ref{H2AZ}), 
\[
\Z_2^{\Omega(A)}\xlongrightarrow{\phi\otimes\mathrm{id}_{\Z_2}} H_2(A;\Z)\otimes\Z_2\xlongrightarrow{p_*\otimes\mathrm{id}_{\Z_2}} H_2(W;\Z)\otimes\Z_2,
\]
where $\Z_2^{\Omega(A)}$ is the elementary abelian $2$-group generated by $\Omega(A)$ with rank $\#\Omega(A)\leq p(\G)+q(\G)$ and $H_2(W;\Z)\otimes\Z_2\cong\Z_2^{p(\G)+q(\G)}$ (Theorem \ref{How}). Since tensoring with $\Z_2$ preserves surjectivity, this forces $\#\Omega(A)=p(\G)+q(\G)$ and both maps are in fact isomorphisms. Thus
$H_2(A(\G);\Z)\otimes\Z_2\cong\Z_2^{p(\G)+q(\G)}$. On the other hand, we have the following exact sequence by universal coefficient theorem,
\[
0\to H_2(A;\Z)\otimes\Z_2\to H_2(A;\Z_2)\to \mathrm{Tor}(H_1(A;\Z),\Z_2)\to 0,
\]
where $\mathrm{Tor}(H_1(A(\G);\Z),\Z_2)=0$ since $H_1(A(\G);\Z)$ is torsion free (Theorem \ref{CE} and the Remark following). Now we conclude $H_2(A(\G);\Z_2)\cong\Z_2^{p(\G)+q(\G)}$ and finish the proof of Theorem \ref{H2Z2}.
\end{proof}

As a byproduct of the proof, we have the following corollaries. Recall that $M(\G)$ is the complement of the complexified arrangement of reflection hyperplanes associated to the Coxeter group $W(\G)$. The orbit space $N(\G)=M(\G)/W(\G)$ has fundamental group $\pi_1(N(\G))\cong A(\G)$. Let $c:N(\G)\to K(A(\G),1)$ be the classifying map. Then $c$ always induces an isomorphism $c_*:H_1(N(\G);\Z)\to H_1(A(\G);\Z)$ and an epimorphism $c_*:H_2(N(\G);\Z)\to H_2(A(\G);\Z)$. We give a sufficient condition on $\G$ such that $c$ induces an isomorphism $c_*:H_2(N(\G);\Z)\to H_2(A(\G);\Z)$.

\begin{cor}\label{cor}
If $\G$ satisfies the following conditions
\begin{itemize}
\item $P(\G)/\sim$ consists of torsion classes. 
\item $\G=\G_{odd}$.
\item $\G$ is a tree.
\end{itemize}
Then 
\[
H_2(A(\G);\Z)\cong\Z_2^{p(\G)}. 
\]
Hence $c$ induces an isomorphism $c_*:H_2(N(\G);\Z)\to H_2(A(\G);\Z)$.
\end{cor}
\begin{proof}
Since $N(\G)$ is path-connected and has fundamental group $\pi_1(N(\G))\cong A(\G)$, there is an exact sequence (see for example Section II.5 Theorem 5.2 of \cite{Brown1982}),
\begin{equation}\label{Hurewicz}
\pi_2(N(\G))\xlongrightarrow{h_2} H_2(N(\G);\Z)\xlongrightarrow{c_*} H_2(A(\G);\Z)\to 0,
\end{equation}
where $h_2$ is the Hurewicz homomorphism. Suppose that $\G$ satisfies the three conditions, then $q_1(\G)=q_2(\G)=q_3(\G)=0$. Theorem \ref{CE} implies that $H_2(N(\G);\Z)\cong\Z_2^{p(\G)}$. Then by Theorem \ref{epi}, $H_2(A(\G);\Z)$ sits in the following sequence
\[
\Z_2^{p(\G)}\twoheadrightarrow H_2(A(\G);\Z)\twoheadrightarrow\Z_2^{p(\G)},
\]
the composition must be an isomorphism, hence $H_2(A(\G);\Z)\cong\Z_2^{p(\G)}$. As a result, $c_*$ must be an isomorphism.
\end{proof}
\begin{cor}\label{intepi}
	If the three conditions in Corollary \ref{cor} are satisfied, then $p:A\to W$ induces an isomorphism
	\[
	p_*:H_2(A;\Z)\to H_2(W;\Z).
	\]
\end{cor}
\begin{proof}
	The corollary follows from Howlett's Theorem \ref{How}, Theorem \ref{epi} and Corollary \ref{cor}.
\end{proof}

\begin{cor}
	For any Coxeter graph $\G$, the induced map $c_*:H_2(N(\G);\Z)\to H_2(A(\G);\Z)$ becomes an isomorphism after tensoring with $\Z_2$.
\end{cor}
\begin{proof}
By right-exactness of tensor functor, taking tensor product with $\Z_2$ preserves the exactness of (\ref{Hurewicz}),
\[
\pi_2(N(\G))\otimes\Z_2\xlongrightarrow{h_2\otimes\mathrm{id}_{\Z_2}} H_2(N(\G);\Z)\otimes\Z_2\xlongrightarrow{c_*\otimes\mathrm{id}_{\Z_2}} H_2(A(\G);\Z)\otimes\Z_2\to 0.
\]
Note that $c_*\otimes\mathrm{id}_{\Z_2}$ is an isomorphism as a consequence of Theorem \ref{H2Z2} and Clancy-Ellis' Theorem \ref{CE}.
\end{proof}

\begin{ex}
	The Coxeter graphs of affine type $\widetilde{D_n}~(n\geq 4)$ and $\widetilde{E_i}~(i=6,7,8)$ all satisfy the conditions in Corollary \ref{cor}. Therefore we compute the second integral homology of the associated Artin groups as follows.
\begin{eqnarray*}
	H_2(A(\widetilde{D_n});\Z)\cong
\begin{cases}
\Z_2^6, &n=4;\cr
\Z_2^3, &n\geq 5.
\end{cases}
\end{eqnarray*}
\[
H_2(A(\widetilde{E_i});\Z)\cong\Z_2, i=6,7,8.
\]
Besides the above cases, the Coxeter graphs of certain hyperbolic Coxeter groups also provide plenty of examples satisfying the conditions in Corollary \ref{cor}. We point out that to the best of the authors' knowledge, the $K(\pi,1)$ conjecture has not been proved in the above mentioned cases.
\end{ex}

\subsection{Homological stability}
We mention a corollary concerning homological stability in the end of this paper. Consider a family of Coxeter graphs $\{\G_i\}_{i\geq1}$, starting from $\G_1$ with a base vertex $s_1$ and each $\G_i ~(i\geq2)$ is obtained by adding a vertex $s_i$ connected to $s_{i-1}$ by an unlabeled edge. The embedding $\G_i\hookrightarrow\G_{i+1}$ of Coxeter graphs induces inclusion of Coxeter groups $W(\G_i)\hookrightarrow W(\G_{i+1})$, as well as inclusion of Artin groups $A(\G_i)\hookrightarrow A(\G_{i+1})$ (cf. \cite{Lek1983,Paris1997}). It is known that the families of Artin groups $\{A(A_n)\}$, $\{A(B_n)\}$ and $\{A(D_n)\}$ possess integral cohomological stability (\cite{Arnold1969,DeConcini1999}). Hepworth proved a more general result for Coxeter groups.
\begin{thm}[\cite{Hepworth2016}]\label{Hep}
	The map $H_k(W(\G_{n-1}))\to H_k(W(\G_n))$ is an isomorphism for $2k\leq n$ with arbitrary constant coefficient.
\end{thm}
As for the sequence of Artin groups $\{A(\G_i)\}$, it is not difficult to see that the first integral homology admits stability. We prove a stability result for the second mod $2$ homology of the sequence $\{A(\G_i)\}$.
\begin{thm} 
	The map $H_2(A(\G_{n-1});\Z_2)\to H_2(A(\G_n);\Z_2)$ is an isomorphism for $n\geq4$.	
\end{thm}
\begin{proof}
	Consider the commutative diagram
\begin{equation}\label{stab}
\begin{tikzpicture}
\matrix(m)[matrix of math nodes,
row sep=2.6em, column sep=2em,
text height=1.5ex, text depth=0.25ex]
{H_2(A(\G_{n-1});\Z_2)&H_2(A(\G_{n});\Z_2)\\
H_2(A(\G_{n-1});\Z)\otimes\Z_2 & H_2(A(\G_{n});\Z)\otimes\Z_2\\
H_2(W(\G_{n-1});\Z)\otimes\Z_2 &H_2(W(\G_{n});\Z)\otimes\Z_2\\};
\path[->,font=\scriptsize,>=angle 90]
(m-1-1) edge node[auto] {} (m-1-2)
(m-2-1) edge node[auto] {} (m-1-1)
(m-2-2) edge node[auto] {} (m-1-2)
(m-2-1) edge node[auto] {} (m-2-2)
(m-2-1) edge node[auto] {$p_*\otimes\mathrm{id}_{\Z_2}$} (m-3-1)
(m-2-2) edge node[auto] {$p_*\otimes\mathrm{id}_{\Z_2}$} (m-3-2)
(m-3-1) edge node[auto] {} (m-3-2);
\end{tikzpicture}
\end{equation}
where the commutativity of the upper square follows from the naturality of the universal coefficient theorem and the lower from tensoring with $\Z_2$ to the following commutative diagram
\begin{equation}\label{intcom}
	\begin{tikzpicture}
\matrix(m)[matrix of math nodes,
row sep=2.6em, column sep=2em,
text height=1.5ex, text depth=0.25ex]
{H_2(A(\G_{n-1});\Z)&H_2(A(\G_{n});\Z)\\
H_2(W(\G_{n-1});\Z)&H_2(W(\G_{n});\Z)\\};
\path[->,font=\scriptsize,>=angle 90]
(m-1-1) edge node[auto] {} (m-1-2)
        edge node[auto] {$p_*$} (m-2-1)
(m-1-2) edge node[auto] {$p_*$} (m-2-2)
(m-2-1) edge node[auto] {} (m-2-2);
\end{tikzpicture}
\end{equation}
Since all vertical maps in (\ref{stab}) are isomorphisms and the bottom horizontal map is an isomorphism when $n\geq 4$ (Theorem \ref{Hep}), the top horizontal map is an isomorphism when $n\geq4$.
\end{proof}
	
\begin{cor}
	It $\G$ satisfies the three conditions in Corollary \ref{cor}, then the map $H_2(A(\G_{n-1});\Z)\to H_2(A(\G_{n});\Z)$ is an isomorphism for $n\geq4$.
\end{cor}
\begin{proof}
	The corollary follows from Corollay \ref{intepi}, Theorem \ref{Hep} and the commutative diagram (\ref{intcom}).
\end{proof}

\section*{Acknowledgement} The first author was partially supported by JSPS KAKENHI Grant Number 26400077. The second author was supported by JSPS KAKENHI Grant Number 16J00125.

%
%
%
%

\bibliographystyle{hep}
\bibliography{gtlatex.tem}

\end{document}